\RequirePackage{fix-cm}
\RequirePackage{amsmath}
\documentclass[smallextended]{svjour3}       
\smartqed  

\usepackage{graphicx}
\usepackage{amsfonts}
\usepackage{amssymb}
\usepackage{latexsym}
\usepackage{mathtools}
\mathtoolsset{showonlyrefs}
\usepackage{listings}
\usepackage{graphicx}
\usepackage{color}
\usepackage{xcolor}
\usepackage{fullpage}
\usepackage{hyperref}
\usepackage{braket}
\usepackage{mathrsfs} 
\usepackage{bm}

\newcommand{\mf}{\mathfrak}
\newcommand{\serie}[1]{\{#1_{n}\}_n}

\def\C{\mathbb{C}}

\def\rank{\mathrm{rank}}

\newcommand{\lowbiglquote}[1][18]{%
   \setbox0=\hbox{\fontsize{#1}{0}\selectfont``}%
   \setlength{\dimen0}{\ht0 - \heightof{A}}%
   \noindent\llap{\smash{\lower\dimen0\box0 }}}
 
\newcommand{\lowbigrquote}[1][18]{%
   \setbox0=\hbox{\fontsize{#1}{0}\selectfont''}%
   \setlength{\dimen0}{\ht0 - \heightof{A}}%
   \unskip\rlap{\smash{\lower\dimen0\box0 }}}

\newcommand{\vf}{\varphi}
\newcommand{\ve}{\varepsilon}

\newcommand{\f}{\mathbb}
\newcommand{\cu}{\subseteq}
\newcommand{\appl}[2]{\langle #1,#2\rangle}

\spnewtheorem{theorem}{Theorem}[section]{\bfseries}{\itshape}

\spnewtheorem{lemma}[theorem]{Lemma}{\bfseries}{\itshape}
\spnewtheorem{definition}[theorem]{Definition}{\bfseries}{\itshape}
\spnewtheorem{remark}[theorem]{Remark}{\bfseries}{\upshape}
\spnewtheorem{example}[theorem]{Example}{\bfseries}{\upshape}
\spnewtheorem{corollary}[theorem]{Corollary}{\bfseries}{\itshape}
\spnewtheorem{proposition}[theorem]{Proposition}{\bfseries}{\itshape}
\spnewtheorem{algorithm}[theorem]{Algorithm}{\bfseries}{\upshape}

\numberwithin{theorem}{section}

\begin{document}
	
	\title{The limit empirical spectral distribution of complex matrix polynomials}
	
	\titlerunning{The empirical spectral distribution of complex matrix polynomials}        
	
	\author{Giovanni Barbarino        \and
		Vanni Noferini 
	}
	
	\institute{G. Barbarino \at
		Department of Mathematics and Systems Analysis, Aalto University, Finland.
		\email{giovanni.barbarino@aalto.fi}    
		\and
		V. Noferini \at
		Department of Mathematics and Systems Analysis, Aalto University, Finland.
		\email{vanni.noferini@aalto.fi}  
	}
	
	\date{Preprint of an article published in Random Matrices: Theory and Applications, 
		https://doi.org/10.1142/S201032632250023X
		 © [copyright World Scientific Publishing Company] https://www.worldscientific.com/worldscinet/rmta
	}

	\maketitle
	
	\begin{abstract}
		We study the empirical spectral distribution (ESD) for complex $n \times n$ matrix polynomials of degree $k$ under relatively mild assumptions on the underlying distributions, thus highlighting universality phenomena. In particular, we assume that the entries of each matrix coefficient of the matrix polynomial have mean zero and finite variance, potentially allowing for distinct distributions for entries of distinct coefficients. We derive the almost sure limit of the ESD in two distinct scenarios: (1)  $n \rightarrow \infty$ with $k$ constant and (2) $k \rightarrow \infty$ with $n$ bounded by $O(k^P)$ for some $P>0$; the second result additionally requires that the underlying distributions are continuous and uniformly bounded. Our results are universal in the sense that they depend on the choice of the variances and possibly on $k$ (if it is kept constant), but not on the underlying distributions. The results can be specialized to specific models by fixing the variances, thus obtaining matrix polynomial analogues of results known for special classes of scalar polynomials, such as Kac, Weyl, elliptic and hyperbolic polynomials.  
		
	\end{abstract}

\section{Introduction}\label{sec:introduction}

Three famous classes of classic results in probability theory concern the distribution of the roots of random scalar polynomials and the distribution of the eigenvalues of random matrices or random pencils. For example, it is known that, when we take the limit $k\to\infty$ for degree $k$ random polynomials with i.i.d.\ coefficients, we find that the roots tend to concentrate uniformly on the unit circle.  On the other hand, when the entries of an $n \times n$ matrix are i.i.d.\ random variables with mean $0$ and variance $n^{-1}$, and in the limit $n \rightarrow \infty$, then the eigenvalues are uniformly distributed on the unit disk. The case of of a random pencil has also been studied, at least in the case of normally distributed entries: under this assumption, the distribution is uniform on the Riemann sphere.

There exist plenty of results on the asymptotic eigenvalue distribution of random matrices (for example 
\cite{BO11,BCC08,LN09,DS07,GK05,GO04,TaoBook,TV08,TVK10,TI13}%
)
and on the  asymptotic root distribution of random polynomials (for example
\cite{BK19,BS86,BBL92,IZ13,IZ97,KZ14,PY15,SV95,TV14}%
). Much effort has also been devoted to studying the phenomenon of universality, i.e., proving that limit distributions are independent of the exact distributions of the coefficients or the entries.

 Eigenvalues of matrix polynomials can be seen as generalizations both of roots of scalar polynomials and of eigenvalues of matrices and pencils. They are of interest in many application, especially in engineering. Yet, the eigenvalue distribution of random matrix polynomial has not received as much attention except in the extremal cases discussed above, i.e., when either the size is $1$ (scalar polynomials) or the degree is $1$ (matrices or pencils). The main motivation for this paper is to fill this gap by considering the general case of $n \times n$ random matrix polynomials of degree $k$. We obtain results that we believe to be interesting per se in the context of random matrix theory, and in particular we find that the limit empirical distributions associated to polynomial eigenvalue problems exhibit universality, as their classical eigenvalue problem counterparts. Moreover,  they may also have some practical relevance, e.g., in the design of numerical experiments based on random coefficients when testing algorithms for solving polynomial eigenvalue problems.

To be more concrete, recall that a square matrix polynomial is an algebraic expression of the form
\begin{equation}\label{eq:matpoly}
P(z) = \sum_{j=0}^k C_j z^j
\end{equation}
where $C_0,C_1,\dots,C_k \in \C^{n \times n}$.
A finite eigenvalue of $P(z)$ is \cite{Dopico2018,LN20} a number $\lambda \in \C$ such that
\[  \rank _\C P(\lambda) <  \rank_{\C(z)} P(z). \] The polynomial eigenvalue problem (PEP) is to  find all such eigenvalues \cite{AT12,DLPVD18,Dopico2018,GLRBook,GT17,LN20,NP15,TM01}, and possibly other related objects (e.g. eigenspaces, minimal indices, infinite eigenvalues) that are not relevant to this paper. PEPs arise in several areas of applied and computational mathematics, including but not limited to control theory, structural engineering, vibrating systems, fluid mechanics, and acoustics \cite{GLRBook,GT17,TM01}.  If the characteristic polynomial $\det P(z)$ is not the zero polynomial, then the finite eigenvalues of $P(z)$ coincide with the roots of $\det P(z)$. A sufficient condition for this to happen is that $\det C_k \neq 0$: in this case, there are exactly $nk$ finite eigenvalues when counted with multiplicities. Under these assumptions, the Empirical Spectral Distribution (ESD) of the matrix polynomial $P(z)$ is defined \cite{BV20} as the atomic probability measure
\[
\mu_P = \frac 1{nk}\sum_{i=1}^{nk} \delta_{\lambda_i},
\]
where $\lambda_1,\dots,\lambda_{nk}$ are the finite eigenvalues of $P(z)$ counted with their multiplicity and $\delta_{\lambda_i}$ is the atomic probability measure with a single atom at $\lambda_i$. 
When we introduce randomicity in the coefficients $C_j$, the ESD becomes a random variable with values in the space of (atomic) probability measures, and one can investigate the average of $\mu_P$, or whether they converge to some known distribution when the size $n$ or the degree $k$ tend to infinity. 

Combining matrix polynomials with probability theory is not a complete novelty. For instance, probabilistic methods have been used to analyze the condition number of PEPs. The expected condition number of Gaussian random PEPs was computed by Armentano and Beltr\'{a}n  \cite{AB19} in the case of complex coefficients and by Beltr\'{a}n and Kozhasov \cite{BK19} for real coefficients. Lotz and Noferini \cite{LN20} performed a probabilistic condition analysis of PEPs by imposing that perturbations are uniformly distributed on the unit sphere. On the other hand, some previous results on the distribution of the eigenvalues preceded the ones derived in this paper. As mentioned above, the average eigenvalue distribution for Gaussian pencils has been derived independently by many authors, e.g., \cite{BA12,EKS94,houghzeros}. Pagacz and Wojtylak \cite{PW20} assessed how the spectral distribution of random matrix polynomials changes under a low rank perturbation via the resolvent method. Finally, the same authors of the present paper were recently able to obtain the asymptotic distribution of monic Gaussian complex matrix polynomials \cite{BV20}.


Here, we address instead non-monic matrix polynomials, and this allows us to greatly relax the assumptions on the underlying distribution with respect to \cite{BV20}, thus obtaining results with crisp flavours of universality. Indeed, we can even allow the distribution of the coefficient $C_j$ to depend on the index $j$. Hence, we derive fairly general results that can be specialised to different types of random matrix polynomials, such as Kac, Weyl, elliptic and hyperbolic polynomials. 
In particular we show that, at least if the distributions involved are continuous and bounded (this assumption is actually only needed when studying the limit $k \rightarrow \infty$),  the limit distributions can be computed whenever the matrix-valued random variables involved have zero mean and finite variance, and the result only depends on the variances and possibly on $k$ (if $k$ is kept constant while $n \rightarrow \infty$).
We express 
the limit empirical eigenvalue distribution of  square matrix polynomials of size $n$ and degree $k$, with all coefficient being complex random matrices, 
in two different regimes: when $n \rightarrow \infty$ with $k$ constant, and when $k \rightarrow \infty$ with $n$ bounded by $k^P$ for some $P>0$ (which in particular includes the case where $n$ is constant)

Below we state the two main results of the paper.
\begin{theorem} \label{thm:esdmonicninf}
	Consider a random vector $[X_0, X_1, \dots, X_k]$ where all $X_j$ are independent random variables with zero mean and unit variance, not necessarily with the same distribution.
	Let $\alpha_0,\dots,\alpha_k$ be complex constants with $\alpha_k\ne 0$. 
	Suppose $C_j$ are $n\times n$ random matrices whose entries are independent copies of $X_j$, and construct the  random matrix polynomial
	\begin{equation}
	P_n(z) =\sum_{j=0}^{k} C_j \alpha_j z^j.
	\end{equation} 
	Call $p$ the smallest index such that $\alpha_p\ne 0$. 	
	Then, for $n \rightarrow \infty$, the empirical spectral distribution of $P_n(z)$ converges almost surely to 
	$$
	\frac pk \delta_0 + \frac{k-p}{k} \mu
	$$
	where $\delta_0$ is the atomic probability measure with a single atom at zero, and $\mu$ is the probability measure with density 	\begin{equation*}
	f(z) = 
	\frac 1{4k\pi}\Delta_z\ln\left(  \sum_{i=p}^{k} |\alpha_i|^2 |z|^{2i-2p} \right).
	\end{equation*}
\end{theorem}

\begin{theorem}\label{thm:esdmonicninfka}
	Suppose that for any $k$ we have a vector of independent random variables $X_0^{(k)}, X_1^{(k)}, \dots, X_k^{(k)}$ with zero mean, unit variance, 
	and continuous distributions with densities uniformly bounded by a constant $M>0$ not depending on $k$.
	Let also $\alpha_0^{(k)}, \alpha_1^{(k)}, \dots, \alpha_k^{(k)}$
	be sequences of complex numbers, where $\alpha_k^{(k)}\ne 0$ for any $k$. We consider the $n \times n$ matrix polynomial of degree $k$
	\begin{equation}
	P_{n,k}(x) =\sum_{j=0}^{k}\alpha_j^{(k)} C_j^{(k)} x^j,
	\end{equation}
	where, for $j=0,\dots,k$ every coefficient $C_j^{(k)}$ is an $n \times n$ random matrix whose entries are i.i.d. copies of $X_j^{(k)}$. 
	If $n = n(k)=O(k^P)$ for some $P>0$, then the empirical spectral distribution of the random matrix polynomials $P_{n,k}$
	converge almost surely to a probability measure  $\mu$ as $k\to\infty$ when the coefficients $\alpha_j^{(k)}$ satisfy the following conditions:
	\begin{itemize}
		\item $\alpha_k^{(k)}\ne 0$ for every $k$,
		\item for almost every $z$, there exists the limit
		$$U(z) = \lim_{k\to\infty} \frac 1{2k} \ln \left(  \sum_{i=0}^k 	\frac{		|\alpha_i^{(k)}|^2}{|\alpha_k^{(k)}|^2}|z|^{2i}  \right)
		$$
		as a $L^1_{loc}$ function or a distribution,
		\item the probability measure $\mu$ satisfies
		$$ 
		\Delta_z U(z) = 2\pi\mu 
		$$
		in the sense of distributions.
	\end{itemize}
\end{theorem}

We emphasize that our results can equivalently be interpreted as results on the empirical eigenvalue distribution of certain structured random pencils: indeed, given a  matrix polynomials $P(z)$, a \emph{linearization} of $P(z)$ is a linear matrix polynomial whose eigenvalues (as well as their geometric and algebraic multiplicities) coincide with those of $P(z)$. Matrix polynomials admit many linearizations: see, e.g., \cite{DLPVD18,GLRBook,NP15} and the references therein. One classic and very important linearization is known as the companion pencil, and it  plays an important role in this paper. The other main tool used is the logarithmic potential of the various spectral measures involved, and the related results coming from distribution theory.

The structure of the paper is as follows. 
In Section \ref{sec:preliminaries} we review some necessary background material on matrix polynomial theory, space of distributions, logarithmic potential, probability theory, and random matrix theory. Moreover, we recall the definition of the empirical spectral distribution of a random matrix polynomial with invertible leading coefficient (Definition \ref{def:ESDP}). In Section \ref{sec:n} and Section \ref{sec:k}, we prove our main results, and apply them to classical families of random polynomials. In Section \ref{sec:concl} we draw some conclusions and comment on possible future research directions. With the goal of improving the readibility of the main part of the paper, the proofs of some technical results, needed in Sections \ref{sec:n} and \ref{sec:k}, are postponed to the appendices.

\section{Mathematical background}\label{sec:preliminaries}

Given an $m \times n$ complex matrix $X$, we denote its singular values by $\sigma_1(X) \geq \dots \geq \sigma_{\min}(X) \geq 0$, having introduced the shorthand $\sigma_{\min}(X) :=\sigma_{\min(m,n)}(X) $. The spectral norm of $X$ is denoted by $\| X \| :=\sigma_1(X)$, while the Frobenius norm of $X$ is $\| X \|_F$.

\subsection{Matrix polynomial theory}

Let $P(z)$ be the matrix polynomial defined in \eqref{eq:matpoly}. We review here those notions in the spectral theory of square complex matrix polynomials that are crucial to this paper. Readers interested in more details on matrix polynomial can find further informations in~\cite{AT12,Dopico2018,GLRBook,LN20} and the references therein.
An element $\lambda \in \C$ is said to be a finite eigenvalue of $P(z)$ if
\begin{equation*}
\rank_{\C}(P(\lambda)) < \rank_{\C(x)}(P(z)) =: r,
\end{equation*}
where $\C(z)$ is the field of rational functions with complex coefficients.

If the leading coefficient $C_k$ of the matrix polynomial $P(z)$ in \eqref{eq:matpoly} is invertible, then $P(z)$ is called \textit{regular}
 and it has $kn$ finite eigenvalues that coincide to the roots of the characteristic equation $\det P(z)=0$, counted with their respective multiplicities.  One can then define the \emph{companion matrix of $P(z)$} as (see e.g.  \cite{AT12})
\begin{equation}\label{eq:companionmatrix}
M=\begin{bmatrix}
-C_k^{-1}C_{k-1} & -C_k^{-1}C_{k-2} & \dots & -C_k^{-1}C_{1} & -C_k^{-1}C_{0}\\
I_n & 0 & 0 & \dots & 0\\
0& I_n & 0 & \dots & 0\\
\vdots & \ddots & \ddots & \ddots & \vdots  \\
0 & \dots & 0 & I_n & 0
\end{bmatrix} \in \C^{kn \times kn},
\end{equation}
where $I_n$ and $0$ are, respectively, the $n\times n$ identity and zero matrices. It can be proved that the eigenvalues of $M$, defined in the classical sense, coincide with the eigenvalues of $P(z)$; moreover, they have the same multiplicities. As a consequence, if the leading coefficient $C_k$ is invertible, the PEP for  $P(z)$ is equivalent to the classical eigenvalue problem for the structured matrix $M$. 
One can also work on the associated \emph{companion pencil} of $P(z)$, defined as $Az-B$, where $A$, $B$ are the  $kn\times kn$ complex matrices defined as
\begin{equation}\label{eq:AB_k}
A:= 
\begin{bmatrix}
C_{k} &  & & \\
&I_n & &  \\
&& \ddots  & \\
&& & I_n  
\end{bmatrix} 
,\qquad 
B:=\begin{bmatrix}
-C_{k-1} & \dots & -C_1 & -C_0\\
I_n & & & \\
& \ddots & & \\
& & I_n & 
\end{bmatrix}.
\end{equation}
Once again, $P(z)$ and $Az-B$ have the same eigenvalues with the same multiplicity, and thus the PEP for $P(z)$ is mathematically equivalent to the generalized eigenvalue problem for $Az-B$. Note that, if $A$ in \eqref{eq:AB_k} is invertible, then $M$ in \eqref{eq:companionmatrix} is well defined and $M=A^{-1}B$.

\subsection{Space of Distributions and Logarithmic Potential}
Let $D(\f C):= C_c^{\infty}(\f C)$ be the space of smooth functions with compact support on $\f C$. We equip it with the following notion of  convergence: a sequence $\serie \vf\cu D(\f C)$ converges to $\vf \in D(\f C)$ if 
\begin{itemize}
	\item the union of the supports of $\vf$ and $\vf_n$ for every $n$ is relatively compact,
	\item $\partial^\alpha \vf_n\to \partial^\alpha \vf$ uniformly for every partial derivative $\partial^\alpha$ of any order.
\end{itemize}
The dual space of $D(\f C)$  is the so-called \textbf{Schwartz Distribution Space} $D'(\f C)$.
Here we report only basic notions and results that we will use in the paper. Interested readers may consult \cite{GR08} for a more thorough review.

From now on, the elements of $C_c^{\infty}(\f C)$ will be called test functions.
For a distribution $T$ and a test function $\vf$, the application of $T$ to $\vf$ is denoted by $T(\vf)$. 
We can embed the space of locally integrable functions $L^1_{loc}(\f C)$ into $D'(\f C)$, since for every $g\in L^1_{loc}(\f C)$, the operator
\[
T_g(\vf) := \appl g \vf = \int_{\f C} g(z)\vf(z)\, {{\rm d}}z
\]
is well-defined, linear and continuous.  It is also easy to see that  distinct (equivalence classes, modulo almost everywhere equality, of) functions in  $L^1_{loc}(\f C)$ produce distinct linear operators. The $L^1_{loc}$ convergence coincides with the pointwise convergence of distributions, as shown by the following lemma.
\begin{lemma}\label{res:L1loc_conv_is_dist_conv}
	Suppose that $g_n$ and $g$ are in $L^1_{loc}(\f C)$. Then $g_n\xrightarrow{L^1} g$ on every compact set if and only if $T_{g_n}\to T_g$ pointwise. In this case, we write $g_n\xrightarrow{D'(\f C)} g$.
\end{lemma}
Suppose now that $\mu$ is a finite mass measure on $\f C$. Call $T_\mu$ the operator
\[
T_\mu(\vf) := \int_{\f C} \vf(z) {{\rm d}}\mu(z).
\]
$T_\mu$ is well-defined, linear and continuous. As a consequence, all finite mass measures are distributions, and we report in the following lemma that the weak convergence of measures coincides with the pointwise convergence of distributions.
\begin{lemma}\label{res:weak_conv_is_dist_conv}
	Suppose that 
	$\mu_n$ and $\mu$ are probability measures on $\f C$.  $T_{\mu_n}\to T_\mu$ pointwise if and only if
	$\mu_n\to \mu$ weakly.  
	In this case, we write $\mu_n\xrightarrow{D'(\f C)} \mu$.
\end{lemma}
Define the Laplace operator on $D'(\f C)$ as follows:
\[
(\Delta T)(\vf) := T(\Delta \vf).
\]
One can see that $\Delta T$ is a well-defined endomorphism of $D'(\f C)$, and that it is linear and continuous.

Assume now  that $\nu$ belongs the space of probability measures such that $z\mapsto \ln|z|$ is integrable in a neighbourhood of infinity. 
Define the \textbf{Logarithmic Potential} of $\nu$, denoted by $U_\nu$, as 
\[
U_\nu(z):= \int_{\f C} \ln|z-z'| \, {{\rm d}}\nu(z').
\]
In this case, 
$U_\nu:\f C\to \f R \cup\{-\infty\}$ since
\[
\int_{|z-z'|> 1} \ln|z-z'| \, {{\rm d}}\nu(z') <\infty.
\]
Since $\ln|\cdot|$ is Lebesgue locally integrable on $\f C$, one can check by using
the Fubini theorem that $U_\nu$ is Lebesgue locally integrable on $\f C$. Moreover, since $\ln|\cdot|$ is the fundamental solution of the Laplace equation in $\f C$, we have the following result.

\begin{lemma}\label{res:laplacian_of_potential_is measure}
	If $z\mapsto \ln|z|$ is integrable in a neighbourhood of infinity with respect to the probability measure $\nu$, then $U_\nu\in L_{loc}^1(\f C)$ is a distribution, and $\Delta U_\nu= 2\pi \nu$.
\end{lemma}

\subsection{Probability theory}

\subsubsection{Notations and a.s. convergence}
In this paper we discuss the almost sure convergence of certain random variables. This subsection collects some standard notion and notation from probability theory. We assume that all the random variables have the same domain $(\Omega,\mf F,\f P)$. If $X$ is a random variable, we informally write, e.g., $X \in S$ to refer to the corresponding event, e.g., $\set{\omega\in \Omega | X(\omega) \in S}$.

\begin{definition}
Given a sequence $\serie Y$ of random variables, and an additional random variable $Y$, with values in the same Hausdorff topological space, we say that $\serie Y$ \textbf{converges almost surely} to $Y$ if
	\[
	\f P\left( \lim_{n\to\infty} Y_n = Y  \right)
	=1 \]
	and in this case we write $Y_n\xrightarrow{a.s.} Y$.
\end{definition}
\noindent If $Y$ is a random variable whose range consists of a single point $p$, we write equivalently
\[
Y_n \xrightarrow{a.s.} Y\quad
\text{ or } \quad 
Y_n \xrightarrow{a.s.} p.
\]

Eventually, we will make use of combinations of random variables with distributions described by bounded density functions. The next lemma is useful to bound the density functions of sums of random variables; the second item appeared in \cite{BSC14}.

\begin{lemma}\label{res:bound_on_density}
	Let $X_i$ be independent random variables with values on $\f C$. Call $M(X)$ the infinity norm of the density of a random variable $X$, where $M(X)=\infty$ if the distribution of $X$ is not continuous. Let $c\in \f C$ be a constant complex number. In this case,
	\begin{enumerate}
		\item $$M(cX) = M(X)/|c|^2,$$
		\item $$ M\left( \sum_{i=1}^{n} X_i \right)^{-1}\ge  \frac 1e \sum_{i=1}^{n} M^{-1}(X_i).   $$
	\end{enumerate}
\end{lemma}

\subsubsection{Empirical spectral distributions}

In this subsection, we review the definitions of empirical spectral distribution for a random matrix and for a random matrix polynomial.

If a deterministic matrix $A\in \C^{m \times m}$ has eigenvalues $\lambda_1(A),\dots,\lambda_m(A)$, its empirical spectral distribution (ESD) is the atomic measure
\[
\mu_A = \frac{1}{m} \sum_{i=1}^{m} \delta_{\lambda_i(A)},
\]
where each distinct eigenvalue appears as many times as its algebraic multiplicity.
Analogously, if $\sigma_1(A)\ge \dots\ge \sigma_m(A)$ are the singular values of $A$,  we define the empirical singular values distribution of $A$ as
\[
\nu_A = \frac{1}{m} \sum_{i=1}^{m} \delta_{\sigma_i(A)}.
\]
A random matrix can be seen as a random variable with values in the appropriate space of matrices, say, $\C^{nk \times nk}$ where $n$ and $k$ are fixed parameters. 
The next classic definition extends the concept of ESDs to random matrices.
\begin{definition}\label{def:ESD}
	Given a random matrix $A(\omega)$ with values in $\C^{m \times m}$, its \textbf{empirical spectral distribution (ESD)} is a random variable with values in the space of probabilities on $\f C$, defined as
	\[
	\mu_A(\omega) := \mu_{A(\omega)} = \frac{1}{m} \sum_{i=1}^{m} \delta_{\lambda_i(A(\omega))},
	\]	
	and analogously, 
	its \textbf{empirical singular values distribution} is a random variable with values in the space of probabilities on $\f R$, defined as
	\[
	\nu_A(\omega) := \mu_{A(\omega)} = \frac{1}{m} \sum_{i=1}^{m} \delta_{\sigma_i(A(\omega))}.
	\]	
\end{definition}

We will study spectral and singular values distributions for some families of random matrices $\serie A$, and find that in our cases the sequence $\{\mu_{A_n}\}_n$ always converges almost surely to a 
limit 
 probability measure $\mu\in \f P(\f C)$. In this case, we simply write
\[
\mu_{A_n}\xrightarrow{a.s.} \mu.
\] 

Finally, we recall the definition of the ESD of a random matrix polynomial \cite{BV20}. We assume that $P(z)$ has size $n \times n$ and degree $k$, and that almost surely  $P(z)$ has invertible leading coefficient. This implies that a.s.\ $P(z)$ has $kn$ finite eigenvalues, that we denote by $\lambda_1(P),\dots,\lambda_{kn}(P)$.
\begin{definition}\label{def:ESDP}
	Let $P(z)$ be a random matrix polynomial of size $n$ and degree $k$, such that its leading coefficient is invertible  almost everywhere. Its \textbf{empirical spectral distribution (ESD)}  is a random variable with values in the space of probabilities on $\C$, defined as
	\[
	\mu_P :=\frac{1}{kn} \sum_{i=1}^{kn} \delta_{\lambda_i(P)}.
	\]	
\end{definition}
It is immediate by Definitions \ref{def:ESD} and \ref{def:ESDP} that the ESD of a random matrix polynomial with invertible leading term coincides with those of its (random) companion matrix \eqref{eq:companionmatrix} and its (random) companion pencil \eqref{eq:AB_k}. We will often rely on these equivalences.

\subsubsection{Bounds on full random matrices}

In the paper, we often deal with $m\times m$  random matrices, with i.i.d.\ entries. The properties of such random matrices have been extensively studied in the literature. Here we collect a list of results that are needed below.\\

\noindent Given a sequence of measures $(\mu_n)_n$ and a function $f:\f C\to \f C$, we say that
\begin{itemize}
	\item $f$ is uniformly bounded (u.b.) if $\limsup_{n\to\infty} \int |f| {{\rm d}}\mu_n <\infty$,
	\item $f$ is uniformly integrable (u.i.) if $\lim_{t\to \infty}\limsup_{n\to\infty} \int_{|f|\ge t} |f| {{\rm d}}\mu_n =0$.
\end{itemize}

\begin{theorem}\label{res:full_random_matrices}
	Let $X$ be a complex random variable with mean $0$ and variance $1$, and for every positive integer $m$ let $A_m$ be the $m \times m$ random matrix whose entries are i.i.d. copies of $X$. The following results hold. 
	\begin{enumerate} 
		\item \cite[Theorem 2.1]{TV08}	With probability $1$, $A_m$ is invertible for all but finitely many $m$.
		\item \cite{BO11} 
		There exists $\beta>0$ for which $ x^\beta+x^{-\beta}$ is u.b.\ and $\ln(x)$
		is u.i.\ for $\{ \nu_{A_m/\sqrt m} \}_m$ almost surely.
		\item \cite{SB95}
		$\{ \nu_{A_m/\sqrt m} \}_m$ converges a.s.\ to the probability distribution $\nu$ (called the Marchenko-Pastur law) on $[0,2]$ with density
		\[
		\frac{\sqrt{4-x^2}}{\pi}.
		\] 
		\item \cite{LN09} 
		Assume that the distribution of $X$ is absolutely continuous with bounded density $f$. In this case,  there exists 	an absolute constant $c>0$  such that for every $u>0$,
		\[
		\f P [\sqrt m \sigma_m(A_m)\le u  ]\le cm^{\frac 32}
		\|f\|_{\infty} u^2 
		\]
		and thus $A_m$ is almost surely invertible.
		\item  For every $v>0$,
		\[
		\f P\left( \|A_m\|\ge v  \right)
		\le  \frac{m^2}{v^2}.
		\]
	\end{enumerate}
	
\end{theorem} 
%

%
Item $2.$ in Theorem \ref{res:full_random_matrices} is especially useful when we need to find some bounds on the logarithmic potentials associated with measures coming from matrices. In the next lemma, this information is required to prove that, for nearby measures, the logarithmic integrals coincide. 
The result is inspired by the ones in \cite{BO11}.

\begin{lemma}\label{res:replacement_false}
	Let $M_n$ and $N_n$ be random complex matrices of size $n\times n$. Suppose that almost surely,
	\begin{enumerate}
		\item $\nu_{M_n} - \nu_{N_n}\to 0$ weakly,
		\item $x\mapsto \ln(x)$ is u.i.\ for $\{\nu_{M_n}\}_n$ and $\{\nu_{N_n}\}_n$.
	\end{enumerate}
	Then, almost surely, $\int \ln(x) {{\rm d}}\nu_{M_n} -
	\int \ln(x) {{\rm d}}\nu_{N_n}\xrightarrow{n\to\infty} 0$.
\end{lemma}
\begin{proof}
	If we fix $\ve >0$, from the second assumption, we can find $t> 0$ large enough so that
	\begin{equation}\label{eq:ln_ui}
	\int_{|\ln(x)|\ge t} |\ln(x)| {{\rm d}}\nu_{M_n } +
	\int_{|\ln(x)|\ge t} |\ln(x)| {{\rm d}}\nu_{N_n } \le \ve
	\end{equation}
	holds almost surely and for large enough $n$. If $\vf:\f R\to\f R$ is a continuous function with support on $[e^{-2t},e^{2t}]$, such that $0\le \vf(x)\le 1$ and $\vf|_{[e^{-t},e^t]} \equiv 1$, then also $\vf(x)\ln(x)$ is continuous and with the same compact support. Moreover, using \eqref{eq:ln_ui}, 
	\begin{align}\label{eq:diff_ln_comp_supp}
	\nonumber
	\left|
	\int_{|\ln(x)|\le t} \ln(x) {{\rm d}}\nu_{M_n } - 
	\int \vf(x)\ln(x) {{\rm d}}\nu_{M_n } 
	\right| &=
	\left|
	\int_{t < |\ln(x)|\le 2t} \vf(x)\ln(x) {{\rm d}}\nu_{M_n } 
	\right|\\
	\nonumber
	&\le \int_{t < |\ln(x)|\le 2t} |\ln(x)| {{\rm d}}\nu_{M_n } \\
	&\le \int_{ |\ln(x)|\ge t} |\ln(x)| {{\rm d}}\nu_{M_n } \le \ve,
	\end{align}
	and the same holds for $\nu_{N_n }$.
	Since $\vf(x)\ln(x)$ is continuous and with compact support, the first assumption tells us that, a.s.\ and for large enough $n$,
	\begin{equation}\label{eq:chissenefregacomelachiamo}
	\left|
	\int \ln(x)\vf(x) {{\rm d}}
	(\nu_{M_n } - \nu_{N_n })
	\right|\le \ve.
	\end{equation}
	Using \eqref{eq:ln_ui}, \eqref{eq:diff_ln_comp_supp} and \eqref{eq:chissenefregacomelachiamo}, we find that
	\begin{align*}
	\left|
	\int \ln(x) {{\rm d}}\nu_{M_n } - 
	\int \ln(x) {{\rm d}}\nu_{N_n } 
	\right|
	&\le 
	\left|
	\int_{|\ln(x)|\le t} \ln(x) {{\rm d}}
	(\nu_{M_n } - \nu_{N_n })
	\right|\\
	& +
	\int_{|\ln(x)|\ge t} |\ln(x)| {{\rm d}}(\nu_{M_n } + \nu_{N_n })\\
	&\le 
	\left|
	\int \ln(x)\mathbf{1}_{|\ln(x)|\le t} - \vf(x)\ln(x) {{\rm d}}\nu_{M_n } 
	\right| \\&+
	\left|
	\int \ln(x)\mathbf{1}_{|\ln(x)|\le t} - \vf(x)\ln(x) {{\rm d}}\nu_{N_n } 
	\right| 
	\\&+ 
	\left|
	\int \ln(x)\vf(x) {{\rm d}}
	(\nu_{M_n } - \nu_{N_n })
	\right|\\
	& +
	\int_{|\ln(x)|\ge t} |\ln(x)| {{\rm d}}(\nu_{M_n } + \nu_{N_n }) \le 4\ve 
	\end{align*}
	for every $\ve$, and in particular, 
	\[
	\int \ln(x) {{\rm d}}\nu_{M_n } - 
	\int \ln(x) {{\rm d}}\nu_{N_n } \to 0
	\]
	almost surely. 
\end{proof}

\section{Proof of Theorem \ref{thm:esdmonicninf}:  Empirical spectral distribution for $n \times n$ complex random matrix polynomials of degree $k$, in the limit $n \rightarrow \infty$}\label{sec:n}

Consider a random vector $[X_0, X_1, \dots, X_k]$ where all $X_j$ are independent random variables with zero mean and unit variance, not necessarily with the same distribution.
Let $\alpha_0,\dots,\alpha_k$ be complex constants with $\alpha_k\ne 0$. 
For all $j$, let $C_j$ be an $n\times n$ random matrix whose entries are independent copies of $X_j$, and build the  random matrix polynomial
\begin{equation}\label{eq:Px}
P_n(z) =\sum_{j=0}^{k} C_j \alpha_j z^j.
\end{equation} 
 Note that we may equivalently have stated that the entries of $C_j$ have variance $|\alpha_j|^2$ to recover the original form \eqref{eq:matpoly}. 
Moreover, observe that the spectrum of $P_{n}(z)$ is invariant with respect to a global scaling $P_n(z) \mapsto \alpha P_n(z)$ ($\alpha \neq 0)$. Hence, from now we take $\alpha_k=1$ without losing generality.

Thanks to item $1.$ in Theorem \ref{res:full_random_matrices}, with probability $1$ and for sufficiently large $n$, $C_k$ is invertible, so
the eigenvalues of $P_n(z)$ coincide with the eigenvalues of its companion matrix $M_n$ \eqref{eq:companionmatrix} and of its companion pencil $A_nz-B_n$ \eqref{eq:AB_k}. (The suffix $n$ in this section emphasizes that we let $n \rightarrow \infty$ while keeping $k$ constant.) 


We are now ready to prove Theorem \ref{thm:esdmonicninf}. Our argument follows the ideas of Girko \cite{GI90} and more recent authors such as Bordenave, Caputo and Chafai \cite{BCC08} or Tao, Vu and Krishnapur \cite{TVK10}. It relies on a few technical lemmata whose proofs are postponed to the appendices.\\

\begin{proof}[Proof of Theorem \ref{thm:esdmonicninf}]
	As discussed above, we can suppose $\alpha_k=1$ because multiplying  the matrix polynomials by a constant does not change their eigenvalues, and the density function in the theorem's statement is not affected. Moreover, since $p$ is the smallest index such that $\alpha_p\ne 0$, we can rewrite the polynomial as
	$$
	P_n(z) =\sum_{j=p}^{k} C_j \alpha_j z^j = z^p \left[
	\sum_{j=0}^{k-p} C_{j+p} \alpha_{j+p} z^j
	\right]
	$$
	and as a consequence, $P_n(z)$ has $np$ zero eigenvalues, and the rest coincide with the eigenvalues of  $Q_n(z):= P_n(z)/z^p$. If the sequence of $Q_n(z)$ admits a constant limit ESD $\mu$, then it is easy to verify that $P_n(z)$ admits
	$$
	\frac pk \delta_0 + \frac{k-p}{k} \mu
	$$
	as limit ESD. Since $Q_n(z)$ still satisfies the hypotheses of the theorem, it is enough to prove the thesis for matrix polynomials with $\alpha_0\ne 0$ and $p=0$. 
 From now on, we thus suppose  $\alpha_k=1$ and $\alpha_0\ne 0$, and we call $A_nz -B_n$ the associated pencil as expressed in \eqref{eq:AB_k}. 
 
 Note that $f(z)$ is a density function on the complex plane, that is, a positive function with unit mass. Indeed, we can compute the Laplacian operator explicitly as
 \begin{equation}\label{eq:limit_density_n}
 	f(z) = \frac 1{4\pi k}\left[	\frac{\sum_{j=1}^k (2j)^2 |\alpha_j|^2 |z|^{2j-2}}{\sum_{j=0}^k |\alpha_j|^2 |z|^{2j}}
 	-
 	\left(
 	\frac{\sum_{j=1}^k 2j |\alpha_j|^2 |z|^{2j-1}}{\sum_{j=0}^k |\alpha_j|^2 |z|^{2j}}
 	\right)^2
 	\right],
 	\end{equation}
and, by the Cauchy-Schwartz inequality, 
		\[ 
		\left(\sum_{j=1}^k (2j)^2 |\alpha_j|^2 |z|^{2j-2}\right)
		\left(\sum_{j=0}^k |\alpha_j|^2 |z|^{2j}\right)
	\ge
	\left(
\sum_{j=1}^k 2j |\alpha_j|^2 |z|^{2j-1}
	\right)^2
	\]
so $f(z)$ is always nonnegative. Via a polar change of coordinates, we can integrate the explicit expression to confirm it has indeed unit mass.
\begin{align*}
\int_{\f C} f(z) {{\rm d}}z
&=\frac 1{4\pi k}
\int_0^{2\pi} \int_0^\infty
	\frac{\sum_{j=1}^k (2j)^2 |\alpha_j|^2 r^{2j-1}}{\sum_{j=0}^k |\alpha_j|^2 r^{2j}}
-r
\left(
\frac{\sum_{j=1}^k 2j |\alpha_j|^2 r^{2j-1}}{\sum_{j=0}^k |\alpha_j|^2 r^{2j}}
\right)^2
 {{\rm d}}r{{\rm d}}\theta \\
&=\frac 1{2 k}
 \int_0^\infty
\frac{\sum_{j=1}^k (2j)^2 |\alpha_j|^2 r^{2j-1}}{\sum_{j=0}^k |\alpha_j|^2 r^{2j}}
-r
\left(
\frac{\sum_{j=1}^k 2j |\alpha_j|^2 r^{2j-1}}{\sum_{j=0}^k |\alpha_j|^2 r^{2j}}
\right)^2
{{\rm d}}r
\\
&=\frac 1{2 k}
\left[
\frac{\sum_{j=1}^k 2j |\alpha_j|^2 r^{2j}}{\sum_{j=0}^k |\alpha_j|^2 r^{2j}}
\right]_0^\infty  = 1.
\end{align*}

Thanks to item $1.$ in Theorem \ref{res:full_random_matrices}, we can work in the probability $1$ space of events where $C_k$, and thus $A_n$, is invertible for any large enough $n$. Observe that $\ln|\cdot|$ is integrable with respect to the measure $\mu_{P_n}$ in a neighbourhood of $\infty$. 
	Let $U_{P_n}(z) := \frac 1n \ln |\det({P_n}(z))|$; clearly $U_{P_n}(z)$ is a function defined for all $z$ but the eigenvalues of $P_n(z)$. 
	Since the leading coefficient of $\det(P_n(z))$, as a polynomial in $z$, is $\det (C_k)$, we have
	\begin{align}\label{eq:UP1}
	\nonumber U_{P_n}(z)
	&= \frac 1n \ln \left| \det(C_k) \prod_{i=1}^{nk}  
	(z - \lambda_i({P_n}))
	\right|\\
	\nonumber&= \frac 1n 
	\ln |\det(C_k)| +
	\frac 1n \sum_{i=1}^{nk}  
	\ln \left|  z - \lambda_i({P_n}) \right|\\
	&= 
	\int_{\f R^+}
	\ln(x) \, {{\rm d}} \nu_{C_k}(x)
	+ k
	\int_{\f C}
	\ln |z -z'|\,  {{\rm d}} \mu_{P_n}(z').
	\end{align}
	On the other hand, evaluating the matrix polynomial at any $z \in \C$, we observe that $P_n(z)$ is a $n\times n$ complex random matrix. As a consequence, the absolute value of its determinant can be expressed as the product of its singular values:
	\begin{align}\label{eq:UP2}
	U_{P_n}(z)
	&= \frac 1n \ln \left|  \prod_{i=1}^{n}  \sigma_i({P_n}(z))
	\right|
	= 
	\int_{\f R^+}
	\ln(x) \, {{\rm d}} \nu_{{P_n}(z)}(x).
	\end{align}
	If we fix $z$ and inspect $P_n(z)$, we find that the entries are i.i.d.\ copies of a random variable with mean $0$ and variance $\sigma^2 = \sum_{i=0}^{k} |z|^{2i}|\alpha_i|^2$, and the distribution does not depend on $n$. As a consequence, for all $n$, $P_n(z)/(\sigma\sqrt n)$ has i.i.d.\ entries with mean 0 and variance $n^{-1}$. 
	From item $2.$ in Theorem \ref{res:full_random_matrices} we know that $x\mapsto \ln(x)$ is u.i. for both the sequences $\left\{ \nu_{\frac {P_n(z)}{\sigma \sqrt n}} \right\}_n$ and $\left\{
	\nu_{\frac{C_k}{\sqrt n}}
	\right\}_n$. Moreover, almost surely both sequences converge weakly to the Marchenko-Pastur law $\nu$ as in Theorem \ref{res:full_random_matrices}-3.
	All the assumptions of Lemma \ref{res:replacement_false} are thus satisfied with
	$M_n = \frac {P_n(z)}{\sigma \sqrt n}$, $N_n = \frac{C_k}{\sqrt n}$, so for almost every $z$ and a.s.\ 
	\[
	\int_{\f R^+}
	\ln(x) \, {{\rm d}} \nu_{\frac {P_n(z)}{\sigma \sqrt n}}(x) - 
	\int_{\f R^+}
	\ln(x) \, {{\rm d}} \nu_{\frac{C_k}{\sqrt n}}(x)
	\xrightarrow{n\to\infty} 0.
	\]
	As a consequence, from \eqref{eq:UP1} and \eqref{eq:UP2}, for almost every $z$ and a.s.\ 
	\begin{align}\label{eq:conv_U}
	\nonumber k
	\int_{\f C}
	\ln |z -z'|\,  {{\rm d}} \mu_{P_n}(z')	
	&=
	\int_{\f R^+}
	\ln(x) \, {{\rm d}} \nu_{{P_n}(z)}(x)
	-
	\int_{\f R^+}
	\ln(x) \, {{\rm d}} \nu_{C_k}(x)\\
	\nonumber&=\ln(\sigma) +
	\int_{\f R^+}
	\ln(x) \, {{\rm d}} \nu_{\frac {{P_n}(z)}{\sigma \sqrt n}}(x) - 
	\int_{\f R^+}
	\ln(x) \, {{\rm d}} \nu_{\frac{C_k}{\sqrt n}}(x)\\
	\implies 	\lim_{n\to\infty} 
	k
	\int_{\f C}
	\ln |z -z'|\,  {{\rm d}} \mu_{P_n}(z')
	&	=
	\frac 12 \ln(\sigma^2)
	= 
	\frac 12 \ln\left(  \sum_{i=0}^{k} |\alpha_i|^2|z|^{2i} \right)
	\end{align}
	The function $\ln(\sigma^2)$ is  continuous  in $z$ since $\alpha_0\ne 0$, so it belongs to $L^1_{loc}(\f C)$ and to the space of Schwartz distributions. 
	If we now call $U(x) = (2k)^{-1} \ln(\sigma^2)$, then by assumption $\Delta_z U(z) = 2\pi \mu$, where $\mu$ has density $f(z)$.
	In Lemma \ref{res:log_uia} we prove that a.s.
	\[
	\limsup_{n\to\infty} \int_{|z|\ge 1} \ln |z| \, {{\rm d}} \mu_{A_n^{-1}B_n}<\infty,
	\]
	where $A_n$, $B_n$ is the pencil associated to our polynomial $P_n(z)$.
	Note that $\mu_{P_n} = \mu_{A_n^{-1}B_n}$, and that 
	from \eqref{eq:conv_U}, $U_{\mu_{P_n}}(z)\to U(z)$ for almost every $z$; we can then apply Lemma \ref{res:conv_log} and conclude that 
	$\mu_{P_n}$ converges almost surely to $\mu$.
\end{proof}

$ $

We stress that Theorem \ref{thm:esdmonicninf} is universal in the sense that it does not depend on the distributions of the random variables $X_j$, as long that they are independent, with zero mean and unit variance. The only parameters actually influencing the limit distributions are the constants $\alpha_j$ and, possibly, the degree $k$. 
Below, by specializing the result to different choices of $\alpha_j$, we obtain the limit ESD for some particular sequences of random matrix polynomials.

\begin{itemize}
	\item \textbf{Kac Polynomials}. If $|\alpha_j|=c^j$ for all $j$ and some $c>0$, then 
	we obtain the limit measures with densities
	\[ 
f(z) = \frac{c^2}{ \pi k} \left(   \frac{1}{(|cz|^2-1)^2}  -  \frac{(k+1)^2|cz|^{2k}}{(|cz|^{2k+2}-1)^2}    \right).
\] 
With a suitable change of  variable, one can derive from this example also the limit ESDs when all $\alpha_j$ are zero, except for $|\alpha_0| = 1$ and  $|\alpha_k|=c^k$. 
When $k=c=1$, then the above formula yields the uniform measure on the Riemann sphere, coinciding with the distribution previously found \cite{BA12,EKS94,houghzeros} for Gaussian pencils (and valid, in that case, also for any finite $n$).
\item \textbf{Binomial or Elliptic Polynomials}. If $|\alpha_j|^2 = {k\choose j}c^{2j}$ for all $j$ and some $c>0$, then the limit ESD has density
$$
f(z) = 
\frac 1{4k\pi}\Delta_z\ln\left(  \sum_{i=0}^{k} {k\choose i} |cz|^{2i} \right) = 
\frac 1{4\pi}\Delta_z\ln\left(  |cz|^2+1 \right) 
= 
\frac{c^2}{\pi(|cz|^2+1)^2}.
$$
	Observe that, in the case of elliptic polynomials, the distribution is independent of $k$ and always coincides with the uniform measure on the Riemann sphere after a stereographic projection with respect to the circle of radius $1/c$.
\item \textbf{Flat or Weyl Polynomials}. If $|\alpha_j|^2 = k^j(j!)^{-1}c^{2j}$, then the limit ESD has density
\[
f(z)=c^2
\frac{
	1 -  (-k|cz|^2 + k+1)k|cz|^{2k} \Gamma^{-1} 
	- k|cz|^{4k+2}\Gamma^{-2}
}{\pi},
\]
where \[
\Gamma =  e^{k|cz|^2}\int_{|cz|^2}^\infty (e^{-s}s)^k\,{{\rm d}}s.
\]
\end{itemize}

\begin{remark}\label{rmk:convergence_of_f_k_for_diverging_k}
	In the cases of Kac and Weyl polynomials, the limit ESDs depend on $k$. 
	If $\mu_k$ is the limit ESD for Kac polynomials with degree $k$, with density
	\[
	\frac{c^2}{ \pi k} \left(   \frac{1}{(|cz|^2-1)^2}  -  \frac{(k+1)^2|cz|^{2k}}{(|cz|^{2k+2}-1)^2}    \right),
	\] 
	then  it is easy to check that, when $k \to \infty$, $\mu_k$ converges weakly to the uniform probability measure on the  circle of radius $1/c$ and centred in zero. 
	
	On the other hand, when
	$\mu_k$ is the limit ESD for Weyl polynomials with degree $k$, with density
	\[
	c^2
	\frac{
		1 -  (-k|cz|^2 + k+1)k|cz|^{2k} \Gamma^{-1} 
		- k|cz|^{4k+2}\Gamma^{-2}
	}{\pi},
	\] 
  it is possible to verify that $\mu_k$ converges weakly to the uniform probability measure on the  disk of radius $1/c$ and centred in zero.
\end{remark}
The observations in Remark \ref{rmk:convergence_of_f_k_for_diverging_k}
suggest that for increasing sequences of degrees and size of matrices, we can also derive a limit ESD. 
In the next section, we analyse the limit ESDs for random matrix and scalar polynomials with increasing degree, and show that for Kac and Weyl polynomials, the computed limits coincide with the ones found in the remark above.

\section{Proof of Theorem \ref{thm:esdmonicninfka}:  Empirical spectral distribution for complex random matrix polynomials of size $n = O(k^P)$, in the limit $k \rightarrow \infty$}\label{sec:k}
Suppose that for any $k$ we have a vector of independent random variables $X_0^{(k)}, X_1^{(k)}, \dots, X_k^{(k)}$ with zero mean, unit variance, 
and continuous distributions with densities bounded by a constant $M>0$ not depending on $k$.
Let also $\alpha_0^{(k)}, \alpha_1^{(k)}, \dots, \alpha_k^{(k)}$
be sequences of complex numbers, where $\alpha_k^{(k)}\ne 0$ for any $k$. We consider the $n \times n$ matrix polynomial of degree $k$
\begin{equation}\label{eq:Pxk}
P_{n,k}(x) =\sum_{j=0}^{k}\alpha_j^{(k)} C_j^{(k)} x^j,
\end{equation}
where, for $j=0,\dots,k$ every coefficient $C_j^{(k)}$ is an $n \times n$ random matrix whose entries are i.i.d. copies of $X_j^{(k)}$. 
Note that one can just assume the entries of $C_j^{(k)}$ to have variance $|\alpha_j^{(k)}|^2$ to find again the original form \eqref{eq:matpoly}.
Since we are interested in the spectrum of $P_{n,k}$, we can always divide by $\alpha_k^{(k)}$, so from now on, we suppose $\alpha_k^{(k)}=1$.
Again, by Theorem \ref{res:full_random_matrices}-4, with probability $1$ the matrix $C_k^{(k)}$ is invertible, and hence
the eigenvalues of $P_{n,k}(x)$ coincide with the eigenvalues of its companion matrix $M_k$ \eqref{eq:companionmatrix} and of its companion pencil $A_kz-B_k$ \eqref{eq:AB_k}. (The suffix $k$ in this section emphasizes that we let $k \rightarrow \infty$ while controlling $n$ by $k^P$ for some constant $P>0$.)

We now prove Theorem \ref{thm:esdmonicninfka}, still following the lead of \cite{BCC08,GI90,TVK10} and thus providing an almost sure limit for the empirical spectral distributions $\mu_{P_{n,k}}$, as in Definition \ref{def:ESDP}, when $k\to\infty$,
and when there exists a constant $P>0$ such that $n = n(k) = O(k^P)$. 
Again, the most technical steps within the proof are dealt with in the appendices, so to improve the readability of the main text.\\

\begin{proof}[Proof of Theorem \ref{thm:esdmonicninfka}]
	For the sake of a lighter notation, let us call $G_k(z) := P_{n(k),k}(z)$, and let us refer to $n(k)$ simply as $n$ when it is not fundamental to stress its dependence on $k$. 
	By Theorem \ref{res:full_random_matrices}-4, we can work in the probability $1$ space of events where $C_k:= C_k^{(k)}$, and thus $A_k$, is invertible. Note that, for a single instance of the polynomial,  $\mu_{G_k}$ belongs to the space of probability measures that integrate $\ln|\cdot|$ in a neighbourhood of $+\infty$. Moreover, we can always suppose $|\alpha_k^{(k)}|=1$ for every $k$, and the assumption of almost everywhere existence for $U(z)$ implies that 
	$$\ln \left(
	\sum_{i = 0, \dots, k-1} \frac{		|\alpha_i^{(k)}|^2}{|\alpha_k^{(k)}|^2}
	\right) =  O(k).$$
	
	Let $U_{k}(z) := \frac 1{nk} \ln |\det(G_k(z))|$ be a function defined for all $z$ except on the eigenvalues of $G_k(z)$. 
	Since the leading term of the polynomial $\det(G_k(z))$ is $\det (C_k)$, we have
	\begin{align}\label{eq:UP1k}
	\nonumber U_{k}(z)
	&= \frac 1{nk} \ln \left| \det(C_k) \prod_{i=1}^{nk}  
	(z - \lambda_i({G_k}))
	\right|\\
	\nonumber&= \frac 1{nk}
	\ln |\det(C_k)| +
	\frac 1{nk} \sum_{i=1}^{nk}  
	\ln \left|  z - \lambda_i({G_k}) \right|\\
	&= 
	\frac 1k
	\int_{\f R^+}
	\ln(x) \, {{\rm d}} \nu_{C_k}(x)
	+ 
	\int_{\f C}
	\ln |z -z'|\,  {{\rm d}} \mu_{G_k}(z').
	\end{align}
	On the other hand, if we fix $z$, then ${G_k}(z)$ is an $n\times n$ complex random matrix. As a consequence, its determinant can be expressed through the product of its singular values:
	\begin{align}\label{eq:UP2k}
	U_{k}(z)
	&= \frac 1{nk} \ln \left|  \prod_{i=1}^{n}  \sigma_i({G_k}(z))
	\right|
	= 
	\frac 1k
	\int_{\f R^+}
	\ln(x) \, {{\rm d}} \nu_{{G_k}(z)}(x).
	\end{align}
	If we fix $z$ and analyse $G_k(z)$, we find that the entries are i.i.d.\ copies of a random variable with mean $0$, variance $\sigma^2 = \sum_{i=0}^k |\alpha_i^{(k)}|^2|z|^{2i}$, and with a continuous density function. As a consequence, $G_k(z)/\sigma$ has i.i.d.\ entries with mean 0 and  variance $1$ for every $k$. Each entry of $G_k(z)/\sigma$ is of the type
	\[
	Y_k = X_k^{(k)} \frac{\alpha_k^{(k)}z^k}{\sigma} +  X_{k-1}^{(k)} \frac{\alpha_{k-1}^{(k)}z^{k-1}}{\sigma} + \dots +  X_{1}^{(k)} \frac{\alpha_1^{(k)}z}{\sigma} + X_0^{(k)} \frac {\alpha_0^{(k)}}\sigma
	\]
	where $X_i^{(k)}$ all have density function uniformly bounded by a constant $M>0$. Thanks to Lemma \ref{res:bound_on_density}, we can thus compute a bound on the density function of every entry $Y_k$  in $G_k(z)/\sigma$.
		$$
		M(Y_k)^{-1} \ge \frac 1e \sum_{i=0}^k M\left(X_i^{(k)} \frac{\alpha_i^{(k)}z^i}{\sigma} \right)^{-1} 
		\ge 
		\frac 1e
		\sum_{i=0}^k \frac{|\alpha_i^{(k)}|^2|z|^{2i}}{\sigma^2}  M^{-1} = 
		\frac 1{Me} 
		$$
	As a consequence
	the density of each entry of $G_k(z)/\sigma$ is bounded by a same absolute constant, that in particular does not depend on $k$, $n$ or $z$. We can now use Theorem \ref{res:full_random_matrices}-4,5 and find that for every $u,v>0$, 
	\[
	\f P \left(\sigma_n(G_k(z)/\sigma)\le u  \right)\le cn^{\frac 52}
 u^2, \qquad 
	\f P\left( \|G_k(z)/\sigma\|\ge v  \right)
	\le  \frac{n^2}{v^2},
	\]
	where $c\in \f R^+$ is an absolute constant.
	Taking  $u = (kn^2)^{-1}$ and $v=kn^2$,
	we find that up to a space of probability $O(1/k^2)$,
	\begin{equation}\label{eq:order_of_convergence_sv_of_P}
	\left| \frac 1{k}
	\int_{\f R^+}
	\ln(x) \, {{\rm d}} \nu_{{G_k}(z)/\sigma}(x) \right|
	\le \frac 1{k} \ln(k n(k)^2)
	= O\left(
	\frac{\ln(k)}{k}
	\right)
	\xrightarrow{k\to\infty} 0,
	\end{equation}	 
	where we used $n(k)= O(k^P)$. 
	Since the sequence $1/k^2$ is summable, 
	Borel-Cantelli Lemma lets us conclude that \eqref{eq:order_of_convergence_sv_of_P} holds for any $z$ and, almost surely, for all $k$ large enough.
	Observe that the very same argument can be repeated after replacing $G_k(z)/\sigma$ by $C_k$. Hence, almost surely,
	\begin{align*}
	&\left| \frac 1{k}
	\int_{\f R^+}
	\ln(x) \, {{\rm d}} \nu_{{G_k}(z)/\sigma}(x) -
	\frac 1{k}
	\int_{\f R^+}
	\ln(x) \, {{\rm d}} \nu_{C_k}(x) \right|& \\
	\le & 
	\left| \frac 1{k}
	\int_{\f R^+}
	\ln(x) \, {{\rm d}} \nu_{{G_k}(z)/\sigma}(x) \right| +
	\left|
	\frac 1{k}
	\int_{\f R^+}
	\ln(x) \, {{\rm d}} \nu_{C_k}(x) \right|
	\xrightarrow{k\to\infty}0.
	\end{align*}
	From \eqref{eq:UP1k} and \eqref{eq:UP2k}
	we thus have that, almost surely,
	\[
	\Bigg| 
	\int_{\f C}
	\ln |z -z'|\,  {{\rm d}} \mu_{G_k}(z')-
	\frac{\ln(\sigma)}{k}
	\Bigg|
	=
	\Bigg| \frac 1{k}
	\int_{\f R^+}
	\ln(x) \, {{\rm d}} \nu_{{G_k}(z)/\sigma}(x) -
	\frac 1{k}
	\int_{\f R^+}
	\ln(x) \, {{\rm d}} \nu_{C_k}(x) \Bigg| \xrightarrow{k\to\infty}0.
	\]
	Note now that, by hypothesis,
	\begin{align*}
	\frac 1{2k} \ln(\sigma^2) =
	\frac 1{2k} \ln\left(\sum_{i=0}^{k} |\alpha_i^{(k)}|^2|z|^{2i} \right)
	\to U_{\mu}(z).
	\end{align*}
	 As a consequence,
	for  almost any $z$ and a.s.\
	\begin{align*}\label{eq:UP3}
	U_{\mu_{G_k}}(z) = \int_{\f C}
	\ln |z -z'|\,  {{\rm d}} \mu_{G_k}(z')
	\to U_{\mu}(z).
	\end{align*}
	Observe that $\mu_{G_k} = \mu_{A_k^{-1}B_k}$, 
	where $A_k$, $B_k$ is the pencil associated to our polynomial $G_k(z)$.
	In Lemma \ref{res:log_uik} we proved that a.s.\ 
	\[
	\limsup_{k\to\infty} \int_{|z|\ge 1} \ln |z| \, {{\rm d}} \mu_{A_k^{-1}B_k}<\infty.
	\]	
	Since $U_{\mu_{G_{k}} }(z)\to  U_{\mu}(z)$ for almost every $z$, we can apply Lemma \ref{res:conv_log} and conclude that 
	$\mu_{P_{n(k),k}} $ converges almost surely to $\mu$.

\end{proof}

$ $

As in the previous section, the result is universal in the sense that it does not depend on the distributions of $X_j^{(k)}$, as long that they are independent, have zero mean, unit variance and continuous and uniformly bounded density functions. 
By specializing the result to different choices of the weights $\alpha_j^{(k)}$, we can obtain the limit ESD for particular sequences of random matrix polynomials.

\begin{itemize}
	\item \textbf{Kac Polynomials}.
	If $|\alpha_j^{(k)}|=c^j$ for all $j,k$, where  $c>0$, then 
		\begin{align*}
		\lim_{k\to\infty} \frac 1{2k} \ln \left(  \sum_{i=0}^k 	\frac{		|\alpha_i^{(k)}|^2}{|\alpha_k^{(k)}|^2}|z|^{2i}  \right)
		 &=
		-c + \lim_{k\to\infty}
		\frac 1{2k} \ln\left(\sum_{i=0}^{k} |cz|^{2i} \right)
		= U(z) = 
		\begin{cases}
		-c + \ln(|cz|), & |cz|>1,\\
		-c, & |cz|\le 1,
		\end{cases}
		\end{align*}
		where one can verify that $ U(z)$ is the logarithmic potential of $\bm 1_{S^1/c}$, the uniform measure on the circle of radius $1/c$ centred at zero. 
		
		 The same limit $U(z)$ with $c=1$, holds when we impose $\alpha_j^{(k)} = 0$ for all $j,k$ except $|\alpha_k^{(k)}| = |\alpha_0^{(k)}|  =1$.
	\item \textbf{Binomial or Elliptic Polynomials}. If $
	|\alpha_j^{(k)}|^2 = {k\choose j}c^{2j}$, for all $j,k$, where $c>0$, then 
	\begin{align*}
	\lim_{k\to\infty} \frac 1{2k} \ln \left(  \sum_{i=0}^k 	\frac{		|\alpha_i^{(k)}|^2}{|\alpha_k^{(k)}|^2}|z|^{2i}  \right)
	&=
	-c +\lim_{k\to\infty}
	\frac 1{2k} \ln\left(\sum_{i=0}^{k} {k\choose j}|z|^{2i} \right)
	=-c +\frac 12\ln\left(  |cz|^2+1 \right),
	\end{align*} 
	leading to a limit ESD with density 
	$$
	f(z) 
	= 
	\frac{c^2}{\pi(|cz|^2+1)^2},
	$$
	that coincides with the uniform measure on the Riemann sphere after a stereographic projection with respect to the circle with radius $1/c$.
	\item \textbf{Flat or Weyl Polynomials}.
	If 
	$
	|\alpha_j^{(k)}|^2 = k^jc^{2j}/j!$,
	for all $j,k$, where $c>0$, then
	\begin{align*}
	\lim_{k\to\infty} \frac 1{2k} \ln \left(  \sum_{i=0}^k 	\frac{		|\alpha_i^{(k)}|^2}{|\alpha_k^{(k)}|^2}|z|^{2i}  \right)
	&=-c+
	\lim_{k\to\infty} \frac 1{2k} \ln \left(  \sum_{i=0}^k 	\frac{		k^i/i!
	}
	{
	k^k/k!
	}|cz|^{2i}  \right)
=
-c+
\frac 12 \cdot
\begin{cases}
|cz|^2-1, & |cz|<1,\\
 \ln  |cz|^2, & |cz|\ge 1.
\end{cases}
	\end{align*}
	leading to the uniform measure on the disk of radius $1/c$ and centred in zero.
	\item \textbf{Hyperbolic Polynomials.} If $|\alpha_j^{(k)}|^2 = \Gamma(d+j)/(\Gamma(d)j!)c^{2j}$ where $c,d>0$ and $\Gamma(x) = \int_0^\infty e^{-t}t^{x-1} \, {{\rm d}}t$, then one can verify
	that this case can be reduced to that of Kac polynomials. Hence, the
	 limit ESD is again the measure $\bm 1_{S^1/c}$.
\end{itemize}

$ $\\
We note that all the derived limit measures are independent of $n$, and in particular we showed that the same results  know for scalar random polynomials \cite{TV14} extend to matrix polynomials of any size. Moreover, they coincide with the limit ESDs computed in Remark \ref{rmk:convergence_of_f_k_for_diverging_k}, suggesting the conjecture that they hold for any function $n(k)$, even those not bounded by $k^P$. 

Actually, with slight modifications, one can prove the same result for monic matrix polynomials, meaning $C_k = I_n$ and $\alpha_k^{(k)}=1$, thus generalizing a result of  \cite{BV20}. In particular, it is enough to add the hypothesis
\[
\left(
\sum_{i=0}^{k-1} |\alpha_i^{(k)}|^2|z|^{2i-2k}
\right)^{-1} = O(1) 
\]
for almost every $z$. 
The latter condition is satisfied, for example, when  $\alpha_{k-1}^{(k)} \ge \gamma>0$ for every $k$, and it includes the case $\alpha_i^{(k)} = 1$  and $X_i^{(k)}$ Gaussian discussed in \cite{BV20}.

\begin{remark}
	The hypotheses on the density functions of $X_j^{(k)}$ can actually be relaxed as long as we can ensure some bounds on the density of the entries $Y_k$ in $G_k(z)/\sigma$. For example, in the case $|\alpha_j^{(k)}|=1$, we can see that it suffices to impose the bound on the density only for $X_k^{(k)}$ and $X_0^{(k)}$, since we would have
	$$
	M(Y_k) \le 
	eM
	\min\left\{
		\frac{1-|z|^{-2k-2}}{1-|z|^{-2}}
	,
		\frac{1-|z|^{2k+2}}{1-|z|^{2}}
	\right\}
	$$
	that is bounded for almost any $z$. 
\end{remark}

\section{Further Work}\label{sec:concl}

We have rigorously obtained the limit of empirical spectral distribution for  complex i.i.d. matrix polynomials under rather mild assumptions on the distribution of their entries. 
In future work, it could be of interest to extend our results by considering, for instance, coefficients restricted to be real (and/or otherwise structured),  random variables with non-zero means, and singular matrix polynomials.

In particular, an interesting extension would be to study cases where the matrix polynomials presents infinite eigenvalues with non-zero asymptotic probability, probably leading to measures on the Riemann sphere. This would allow to relax the hypothesis of boundedness for the density functions of the random variables involved in the case $k\to\infty$.

Under suitable assumptions, we note that taking the average in \eqref{eq:UP1k} leads to an expression for the logarithmic potential of the average ESD, so another direction for potential future research is to seek the average ESD for fixed parameter $n,k$. In fact, this has already been done in same special case (\cite{BA12,EKS94,houghzeros} for Gaussian pencils), but never in the same generality as discussed in this paper.

\section*{Acknowledgements}

We acknowledge the computational resources provided by the Aalto Science-IT project. We are grateful to Carlos Beltr\'{a}n for illuminating discussions that inspired us to improve both the presentation and the generality of our results.

\bibliographystyle{abbrv}
\bibliography{randomthings3}

\appendix

\section{Appendix: Matrix Theory}

In this appendix we collect some useful results on eigenvalues and singular values of matrices.

\begin{lemma}[Weyl Inequalities \cite{WE49}]\label{res:Weyl}
	For every $n\times n$ matrix $A$ and any $1\le k\le n$, 
	\[
	\prod_{i=1}^{k} |\lambda_i(A)| \le 
	\prod_{i=1}^{k} \sigma_i(A), \qquad 
	\prod_{i=k}^{n} |\lambda_i(A)| \ge 
	\prod_{i=k}^{n} \sigma_i(A),
	\]
	where
	\[
	|\lambda_1(A)|\ge |\lambda_2(A)|\ge \dots 
	\ge |\lambda_n(A)|,\qquad
	\sigma_1(A)\ge \sigma_2(A)\ge \dots 
	\ge \sigma_n(A).
	\]
\end{lemma}

\begin{lemma}[\cite{HJ91}]\label{res:sv_product_matrices}
	If $M,N$ are square complex matrices, then 
	\[
	\sigma_{2i}({AB}) \le\sigma_i(A)\sigma_i(B),
	\quad \forall i : 2i\le n,
	\qquad 
	 \sigma_{2i-1}(AB)\le \sigma_i(A)\sigma_i(B),
	 \quad \forall i : 2i-1\le n.
	\]
\end{lemma}

\begin{lemma}[\cite{BO11}]\label{res:x^a_ub_product}
	If $M,N$ are square complex matrices, then for any $\alpha>0$ 
	\[
	\int x^\alpha {{\rm d}}\nu_{MN}
	\le 2
	\left( \int x^{2\alpha} {{\rm d}}\nu_{M}  \right)^{1/2}
	\left( \int x^{2\alpha} {{\rm d}}\nu_{N}  \right)^{1/2}.
	\]
\end{lemma}

A result we use in our arguments is the following interlacing property of the singular values of submatrices.
\begin{theorem}[Interlacing Singular Values for Submatrices \cite{THO72}]\label{thm:Interlacing}
	Let $A$ be an $m\times n$ matrix, and $B$ a $p\times q$ submatrix of $A$ with singular values, respectively, 
	\[
	\alpha_1\ge \alpha_2 \ge\dots\ge \alpha_{\min(m,n)}, \qquad
	\beta_1\ge \beta_2 \ge\dots\ge \beta_{\min(p,q)}.
	\]
	Then
	\begin{align*}
	&\alpha_i\ge \beta_i, && i=1,2,\dots,\min(p,q),\\
	&\beta_i\ge \alpha_{i+(m-p)+(n-q)}, && i=1,2,\dots,\min(p+q-n,p+q-m).
	\end{align*}
\end{theorem}

\section{Appendix: Logarithmic Potential}

In this appendix we provide the proofs of some technical results on logarithmic potential of finite measures.

\begin{lemma}\label{res:conv_log}
	Let $\mu$ be a probability measure, and
	$U(z)\in L^1_{loc}(\f C)$ with  $\Delta U = 2\pi\mu$. 
	Suppose that $\serie \mu$ is a sequence of probability measures such that $U_{\mu_n}(z) \xrightarrow{n\to \infty} U(z)$ almost everywhere. If there exists $R\ge 1$ such that 
	\[
	\limsup_{n\to\infty} \int_{|z|\ge R} \ln |z| \, {{\rm d}} \mu_{n}<\infty,
	\]
	then 
	$\mu_{n} \to \mu$ weakly.
\end{lemma}
\begin{proof}
	Observe first that
	\begin{equation}\label{resapp1}
	\limsup_{n\to\infty} \int_{|z|\ge R} \ln |z| \, {{\rm d}} \mu_{n}<\infty
	\implies 
	\sup_{n\ge N} \int_{|z|\ge R} \ln |z| \, {{\rm d}} \mu_{n}<C
	\end{equation}
	for a certain index $N$ and a constant $C>0$. We can  suppose, without loss of generality, that $N=1$. As a consequence, every $\mu_n$ is integrable in a neighbourhood of infinity, and by Lemma \ref{res:laplacian_of_potential_is measure}, the functions $U_{\mu_n}$ are all in $L^1_{loc}(\f C)$.

	Let now $K$ be any relatively compact subset of $\f C$, and in particular suppose that all $z\in K$ satisfy $|z|\le M$ for some $M>1$.  $U_{\mu_n}$  are in $L^1(K)$, so we can estimate their integral. 
	\begin{align*}
	\int_K |U_{\mu_n}(z)| \, {{\rm d}}z  &= 
	\int_K \left|
	\int_{\f C} \ln |z -z'|\,  {{\rm d}} \mu_n(z')
	\right| \, {{\rm d}}z \\
	&=
	\int_K \left|
	\int_{|z -z'|\le 1} \ln |z -z'|\,  {{\rm d}} \mu_n(z')
	+
	\int_{|z -z'|>1} \ln |z -z'|\,  {{\rm d}} \mu_n(z')
	\right| \, {{\rm d}}z\\
	&\le
	\int_K 
	\int_{|z -z'|\le 1} -\ln |z -z'|\,  {{\rm d}} \mu_n(z') \, {{\rm d}}z
	+
	\int_K 
	\int_{|z -z'|>1} \ln |z -z'|\,  {{\rm d}} \mu_n(z')
	\, {{\rm d}}z.\\
	\end{align*}
	In  both terms, we can invoke the Fubini-Tonelli theorem and swap the order of integration. For the first term,
	\[
	\int_K 
	\int_{|z -z'|\le 1} -\ln |z -z'|\,  {{\rm d}} \mu_n(z') \, {{\rm d}}z
	\le
	\int_{\f C} 
	\int_{z\in K, |z -z'|\le 1}
	-\ln |z -z'|\, {{\rm d}}z\,  {{\rm d}} \mu_n(z'),
	\]	
	and switching to polar coordinates,
	\[
	\le \int_{\f C} 
	\int_{|z -z'|\le 1}
	-\ln |z -z'|\, {{\rm d}}z\,  {{\rm d}} \mu_n(z') 
	\le 
	\int_{\f C}
	2\pi 
	\int_{0}^1
	-r\ln r\, {{\rm d}}r\,  {{\rm d}} \mu_n(z') 
	= \frac \pi 2.
	\]
	Actually, if $|K|\le \pi$, one can find a better bound. Indeed, in this case, we do not integrate on the whole unit  disk, but we find that it is dominated by the integral on the disk with radius $\sqrt{|K|/\pi}$, obtaining the bound
	\[
	\le \int_{\f C} 
	\int_{|z -z'|\le \sqrt{|K|/\pi}}
	-\ln |z -z'|\, {{\rm d}}z\,  {{\rm d}} \mu_n(z') 
	\le 
	\int_{\f C}
	2\pi 
	\int_{0}^{\sqrt{|K|/\pi}}
	-r\ln r\, {{\rm d}}r\,  {{\rm d}} \mu_n(z') 
	= \frac {|K|}{2} \left(1 - \ln \frac{|K|}{\pi} \right).
	\]
	For the second term, we can use \eqref{resapp1} so that 
	\begin{align*}
	\int_K 
	\int_{|z -z'|\ge 1} \ln |z -z'|\,  {{\rm d}} \mu_n(z') \, {{\rm d}}z
	&\le 
	\int_{\f C} 
	\int_{z\in K, |z -z'|\ge 1}
	\ln |z -z'|\, {{\rm d}}z\,  {{\rm d}} \mu_n(z') \\ &
	\le 
	\int_{\f C} 
	\int_{z\in K}
	\ln (M +|z'|)\, {{\rm d}}z\,  {{\rm d}} \mu_n(z') \\ &
	= 
	\int_{\f C} 
	|K|
	\ln (M +|z'|)\,  {{\rm d}} \mu_n(z') \\ &
	=
	\int_{|z'|<R} 
	|K|
	\ln (M +|z'|)\,  {{\rm d}} \mu_n(z') 
	+
	\int_{|z'|\ge R} 
	|K|
	\ln (M +|z'|)\,  {{\rm d}} \mu_n(z') \\
	&
	\le 
	|K|
	\ln (M +R)
	+
	\int_{|z'|\ge R} 
	|K|
	\ln ((M/R+1)|z'|)\,  {{\rm d}} \mu_n(z') \\
	&
	\le 
	|K|
	\ln (M +R)
	+
	|K|
	\ln (M/R+1)
	+
	|K|
	\int_{|z'|\ge R} 
	\ln |z'|\,  {{\rm d}} \mu_n(z') \\
	&
	\le 
	|K|(
	2\ln (M +R)
	-\ln (R)
	+
	C).
	\end{align*}
	As a consequence, when $|K|> \pi$ we have that $\int_K |U_{\mu_n}(z)| \, {{\rm d}}z $ is bounded from above by
	\begin{equation}\label{eq:u.i.p}
	\int_K |U_{\mu_n}(z)| \, {{\rm d}}z
	\le 
	\frac \pi 2
	+
	|K|(
	2\ln (M +R)
	-\ln (R)
	+
	C)
	<\infty,
	\end{equation}
	and when $|K|\le \pi$, we get
	\begin{equation}\label{eq:u.i.}
	\int_K |U_{\mu_n}(z)| \, {{\rm d}}z
	\le 
	\frac {|K|}{2} \left(1 - \log \frac{|K|}{\pi} \right)
	+
	|K|(
	2\ln (M +R)
	-\ln (R)
	+
	C)
	\xrightarrow{|K|\to 0}
	0.
	\end{equation}
	If we now fix $\ve >0$, then 
	\begin{align*}
	\int_K |U_{\mu_n}(z)-U(z)| \, {{\rm d}}z &=
	\int_K |U_{\mu_n}(z)-U(z)|\bm 1_{|U_{\mu_n}(z)-U(z)|\le \ve } \, {{\rm d}}z
	+
	\int_K |U_{\mu_n}(z)-U(z)|\bm 1_{|U_{\mu_n}(z)-U(z)|> \ve } \, {{\rm d}}z\\
	&\le 
	|K|\ve 
	+
	\int_K |U_{\mu_n}(z)|\bm 1_{|U_{\mu_n}(z)-U(z)|> \ve } \, {{\rm d}}z
	+
	\int_K |U(z)|\bm 1_{|U_{\mu_n}(z)-U(z)|> \ve } \, {{\rm d}}z.
	\end{align*}
	Define $E_n:= \{z\in K: |U_{\mu_n}(z)-U(z)|> \ve \}$ and observe that in the limit $n \rightarrow \infty$ it must be $|E_n|\to 0$ since $U_{\mu_n}\to U$. 
	Since $U(z)\in L^1(K)$, we get that $\int_K |U(z)|\bm 1_{E_n} \, {{\rm d}}z\xrightarrow{|E_n|\to 0}
	0$. Moreover, 
	since $E_n$ is relatively compact and $|E_n|\to 0$, then
	\eqref{eq:u.i.} yields
	\[
	\int_K |U_{\mu_n}(z)|\bm 1_{E_n } \, {{\rm d}}z
	=
	\int_{E_n} |U_{\mu_n}(z)|\, {{\rm d}}z
	\xrightarrow{|E_n|\to 0}
	0
	\]
	so 
	\[
	\limsup_{n\to\infty}\int_K |U_{\mu_n}(z)-U(z)| \, {{\rm d}}z 
	\le |K|\ve 
	\] 
	for every $\ve>0$, and in particular, $U_{\mu_n}\xrightarrow{L^1(K)}U$.

	To conclude, we note that $U_{\mu_n}$ and $U$ are locally integrable, and thus they are Schwartz distributions. From Lemma \ref{res:L1loc_conv_is_dist_conv}, $U_{\mu_n}\xrightarrow{L^1_{loc}}U$ implies that $U_{\mu_n}\xrightarrow{D'(\f C)}U$. 
	Moreover,
	the Laplacian operator $\Delta$ is continuous on the space of distributions, and 
	since $\mu_n$ are integrable in a neighbourhood of infinity by \eqref{resapp1}, then  $\Delta U_{\mu_n} = 2\pi\mu_n\in D'(\f C)$ from Lemma \ref{res:laplacian_of_potential_is measure}. As a consequence $\Delta U_{\mu_n} = 2\pi \mu_n \xrightarrow{D'(\f C)}2\pi \mu = \Delta U$,
	that yields the weak convergence of measures due to Lemma \ref{res:weak_conv_is_dist_conv}.
\end{proof}

\section{Appendix: Matrix Polynomials}\label{sec:app}
In this appendix we state and prove  two intermediate lemmata that are needed for our main results.
In the proofs, we borrow in part arguments from \cite{BO11}, and we use the following $kn\times kn$ block diagonal matrix $D_n$ as a normalization factor. 
\begin{equation}\label{eq:D}
D_n:=\begin{bmatrix}
\frac 1{\sqrt n}I_n &  &  & \\
&I_n & & \\
&& \ddots & \\
& && I_n  
\end{bmatrix}.
\end{equation}

\begin{lemma}\label{res:log_uia}
	Let $A_n$, $B_n$ be the pencil associated with the random matrix polynomial \eqref{eq:Px}, where $\alpha_k=1$.  We have that 
	\[
	\limsup_{n\to\infty} \int_{|z|\ge 1} \ln |z| \, {{\rm d}} \mu_{A_n^{-1}B_n}<\infty.
	\]
\end{lemma}
\begin{proof}
	Recall that
	\[
	A_n:= 
	\begin{bmatrix}
	C_{k} &  & & \\
	&I_n & &  \\
	&& \ddots  & \\
	&& & I_n  
	\end{bmatrix} 
	,\qquad 
	B_n:=\begin{bmatrix}
	-\alpha_{k-1}C_{k-1} & \dots & -\alpha_1C_1 & -\alpha_0C_0\\
	I_n & & & \\
	& \ddots & & \\
	& & I_n & 
	\end{bmatrix}
	\]
	and that, thanks to Theorem \ref{res:full_random_matrices}-1, we can always consider $A_n$ invertible.\\
	
	\noindent Note that
	thanks to Lemma \ref{res:Weyl}, and for any $\alpha>0$, 
	\begin{align*}
	\int_{|z|\ge 1} \ln |z| \, {{\rm d}} \mu_{A_n^{-1}B_n}
	& = 
	\frac 1{nk} 
	\sum_{|\lambda_i| \ge 1} \ln|\lambda_i(A_n^{-1}B_n)| = 
	\frac 1{nk} 
	\ln \left(
	\prod_{|\lambda_i| \ge 1} |\lambda_i(A_n^{-1}B_n)|
	\right)\\
	& \le 
	\frac 1{nk} 
	\ln \left(
	\prod_{|\lambda_i| \ge 1} \sigma_i(A_n^{-1}B_n)
	\right)  \le 
	\frac 1{nk} 
	\ln \left(
	\prod_{\sigma_i \ge 1} \sigma_i(A_n^{-1}B_n)
	\right)\\
	&= 
	\int_{x \ge 1} \ln x \, {{\rm d}} \nu_{A_n^{-1}B_n} \le c_\alpha + 
	\int_{\f R} x^\alpha \, {{\rm d}} \nu_{A_n^{-1}B_n},
	\end{align*}
	where $c_\alpha$ is a constant that depends only on $\alpha$. 
	Moreover, from Lemma \ref{res:x^a_ub_product},
	\[
	\int x^\alpha \,{{\rm d}}\nu_{A_n^{-1}B_n}
	\le 2
	\left( \int x^{2\alpha} \,{{\rm d}}\nu_{(D_nA_n)^{-1}}  \right)^{1/2}
	\left( \int x^{2\alpha} \,{{\rm d}}\nu_{D_nB_n}  \right)^{1/2}
	= 2
	\left( \int x^{-2\alpha} \,{{\rm d}}\nu_{D_nA_n}  \right)^{1/2}
	\left( \int x^{2\alpha} \,{{\rm d}}\nu_{D_nB_n}  \right)^{1/2}
	\]
	where $D_n$ is the normalization matrix in \eqref{eq:D}. As a consequence, it is enough to prove that there exists an $\alpha>0$ for which $x^{-\alpha}$ is u.b. for $\{ \nu_{D_nA_n} \}_n$ and $x^{\alpha}$ is u.b. for $\{ \nu_{D_nB_n} \}_n$.  Observe that set of the singular values of $D_nA_n$ consists of $nk-n$ singular values equal to $1$ and of the $n$ singular values of $C_k/\sqrt n$. Theorem \ref{res:full_random_matrices}-2 shows that there exists $\alpha>0$ for which  
	\begin{align*}
	\limsup_{n\to\infty}	\int x^{-\alpha} \,{{\rm d}}\nu_{D_nA_n} & =
	\frac 1k \limsup_{n\to\infty}	\int x^{-\alpha} \,{{\rm d}}\nu_{C_k/\sqrt n} + \frac{k-1}{k} <	\infty.
	\end{align*} 
Moreover, for $\alpha=2$, we have that
\begin{align*}
\int x^{2} \,{{\rm d}}\nu_{D_nB_n} & = \frac 1{nk} \|D_nB_n\|_F^2\\
&= \frac {n(k-1)}{nk} + 	\frac 1k \sum_{j=0}^{k-1}
|\alpha_j|^2
\frac 1{n^2}  \|C_j\|_F^2 
\end{align*} 
and by strong law of large numbers, every term $\|C_j\|_F^2 /n^2$ converges almost surely to $1$, so
\[
\lim_{n\to\infty} \int x^{2} \,{{\rm d}}\nu_{D_nB_n}  =
\frac{k-1}k + 	\frac 1k \sum_{j=0}^{k-1}
|\alpha_j|^2<\infty. 
\]
\end{proof}

\begin{lemma}\label{res:log_uik}
	Let $A_k$ and $B_k$ be the 
	pencil associated to the matrix polynomial  \eqref{eq:Pxk}, where $n=n(k)$
	 is a function for which 
	there exists a constant $P>0$ such that $n(k) = O(k^P)$. 
	Moreover, suppose that 
	the coefficients $\alpha_j^{(k)}$ satisfy
		$$\ln \left(
		1 +\sum_{j = 0, \dots, k-1} 	|\alpha_j^{(k)}|^2
		\right) =  O(k).$$
	 Then, almost surely,
	\[
	\limsup_{k\to\infty} \int_{|z|\ge 1} \ln |z| \, {{\rm d}} \mu_{A_k^{-1}B_k}<\infty.
	\]
\end{lemma}
\begin{proof}
	Recall that
	\[
A_k:= 
\begin{bmatrix}
C_{k} &  & & \\
&I_n & &  \\
&& \ddots  & \\
&& & I_n  
\end{bmatrix} 
,\qquad 
B_k:=\begin{bmatrix}
-\alpha^{(k)}_{k-1}C_{k-1} & \dots & -\alpha^{(k)}_{1}C_1 & -\alpha^{(k)}_{0}C_0\\
I_n & & & \\
& \ddots & & \\
& & I_n & 
\end{bmatrix}
	\]
	and due to item $4.$ in Theorem \ref{res:full_random_matrices}, we can always consider $A_k$ invertible. Let
	\[
	D_k:= 
	\begin{bmatrix}
	\frac 1{\sqrt{n}} I_n&  & & \\
	&I_n & &  \\
	&& \ddots  & \\
	&& & I_n  
	\end{bmatrix} 
	\]
	and notice that $A_k^{-1}B_k = (D_kA_k)^{-1} (D_kB_k)$.
	\\
	
	\noindent Using Lemma \ref{res:Weyl} and 
	Lemma
	\ref{res:sv_product_matrices}, for any $\alpha>0$, 
	\begin{align*}
	\int_{|z|\ge 1} \ln |z| \, {{\rm d}} \mu_{A_k^{-1}B_k}
	& = 
	\frac 1{nk} 
	\sum_{|\lambda_i| \ge 1} \ln|\lambda_i(A_k^{-1}B_k)| = 
	\frac 1{nk} 
	\ln \left(
	\prod_{|\lambda_i| \ge 1} |\lambda_i(A_k^{-1}B_k)|
	\right)\\
	& \le 
	\frac 1{nk} 
	\ln \left(
	\prod_{|\lambda_i| \ge 1} \sigma_i(A_k^{-1}B_k)
	\right)  \le 
	\frac 1{nk} 
	\ln \left(
	\prod_{\sigma_i \ge 1} \sigma_i(A_k^{-1}B_k)
	\right)\\
	&\le 
	\frac 2{nk} 
	\ln \left(
	\prod_{\sigma_i(A_k^{-1}B_k)\ge 1} \sigma_i((D_kA_k)^{-1})\sigma_i(D_kB_k)
	\right)\\
	&  \le 
	\frac 2{nk} 
	\ln \left(
	\prod_{\sigma_i \ge 1} \sigma_i((D_kA_k)^{-1})
	\right) 
	+
	\frac 2{nk} 
	\ln \left(
	\prod_{\sigma_i \ge 1} \sigma_i(D_kB_k)
	\right) \\
	&=2 
	\int_{x\ge 1} \ln x \, {{\rm d}} \nu_{(D_kA_k)^{-1}}  
	+2
	\int_{x\ge 1} \ln x \, {{\rm d}} \nu_
	{D_kB_k}\\
	&\le
	2 c_\alpha + 2
	\int_{\f R^+} x^{-\alpha} \, {{\rm d}} \nu_{D_kA_k}  
	+2
	\int_{x\ge 1} \ln x \, {{\rm d}} \nu_
	{D_kB_k},
	\end{align*}
	where $c_\alpha$ is a constant that depends only on $\alpha$. Observe that $D_kB_k$ has a rectangular identity submatrix of size $(k-1)n\times kn$. 
	
	If $S_k=1+\sum_{s = 0, \dots, k-1}
	|\alpha_s^{(k)}|^2$, then by assumption $\ln(S_k) = O(k)$. 
	Thanks to Theorem \ref{thm:Interlacing}, we thus know that all but its largest $n$ singular values are bounded by $1$. We can crudely bound the largest singular values with the spectral norm of $D_kB_k$ by 
	\[
	\f P[\|D_kB_k\| >  e^k\sqrt {S_k}] \le 
	\f P[\|D_kB_k\|_F^2 > e^{2k} {S_k} ]\le 
	\frac {n(k-1) + n S_k  }{e^{2k} {S_k}}.
	\]
Therefore, since $n(k) = O(k^P)$, we have that $\|D_kB_k\| \le e^k\sqrt {S_k}$ up to a space of probability $O(e^{-k})$. 
	As a consequence,
	\[
	\int_{x>1} \ln x \, {{\rm d}}\nu_{D_kB_k}
	=
	\frac{1}{kn}
	\sum_{i=1}^{n} \ln (\sigma_i(D_kB_k))
	\le
	\frac{\ln(e^k\sqrt {S_k})}{k} = O(1).
	\] 
Since $n(k)=O(k^P)$, we can use item $4.$ in Theorem \ref{res:full_random_matrices} to see that
		\[
		\f P [ \sigma_{n}(C_k/\sqrt{n})\le k^{-1-\frac 74P}  ]\le cn(k)^{\frac 72}
		\|f\|_{\infty} k^{-2-\frac 72P} =O(k^{-2}).
		\]
		Thus, up to a space of probability $O(k^{-2})$, 
		\begin{align}\label{eq:bound_A_as}
		\int x^{-\alpha} \,{{\rm d}}\nu_{D_kA_k} & =
		\frac 1k 	\int x^{-\alpha} \,{{\rm d}}\nu_{C_k/\sqrt n} + \frac{k-1}{k} 
		\le
		\frac 1k 	 k^{(1+\frac 74P)\alpha}  + \frac{k-1}{k} \le 2
		\end{align} 
		for any $\alpha \le 1/(1+\frac 74P)$. By Borel-Cantelli Lemma we conclude that
			\begin{align*}
		\int_{|z|\ge 1} \ln |z| \, {{\rm d}} \mu_{A_k^{-1}B_k}
		& \le
		2 c_\alpha + 2
		\int_{\f R^+} x^{-\alpha} \, {{\rm d}} \nu_{D_kA_k}  
		+2
		\int_{x\ge 1} \ln x \, {{\rm d}} \nu_
		{D_kB_k} = O(1)
		\end{align*} almost surely.	
\end{proof}

\end{document}